\documentclass[12pt]{amsart}

\usepackage{amssymb,amsmath,amsthm, url}
\usepackage{comment}
\usepackage{enumitem}
\usepackage[all]{xy}

\numberwithin{equation}{section}

\numberwithin{subsection}{section}

\allowdisplaybreaks[1]

\newtheorem*{namedtheorem}{\theoremname} \newcommand{\theoremname}{testing}

\newtheorem{theorem}{Theorem}[section]
\newtheorem{proposition}[theorem]{Proposition}
\newtheorem{proposition-definition}[theorem] {Proposition-Definition}
 \newtheorem{lemma}[theorem]{Lemma}

\newtheorem{question}[theorem]{Question}

\theoremstyle{definition} 
 
 \newtheorem{remark}[theorem]{Remark}

\newtheorem*{claim}{Claim}

\newcommand\Conj{\operatorname{Conj}}

\newcommand\diag{\operatorname{diag}} 
\newcommand\Char{\operatorname{char}}

\newcommand\Hypersurf{\operatorname{Hypersurf}} 
 
\newcommand\DHyp{\operatorname{DHyp}} 

\newcommand{\GL}{\mathrm{GL}} 
 
\newcommand{\PGL}{\mathrm{PGL}}

\newcommand\bbZ{\mathbb{Z}}

\newcommand\bP{\mathbf{P}}

\newcommand\arr{\ifinner\to\else\longrightarrow\fi}

\newcommand\arrto{\ifinner\mapsto\else\longmapsto\fi}

\def\displaytimes_#1{\mathrel{\mathop{\times}\limits_{#1}}}

\def\displayotimes_#1{\mathrel{\mathop{\bigotimes}\limits_{#1}}}

\newcommand\Span{\operatorname{Span}}

\newcommand\Spec{\operatorname{Spec}}

\newcommand\Ker{\operatorname{Ker}}

\newdir{ >}{{}*!/-5pt/@{>}}

\newcommand\doublelong[2]{\mathbin{\xymatrix{{}\ar@<3pt>[r]^{#1}
\ar@<-3pt>[r]_{#2}&}}}

\newlength{\ignora}

\newcommand{\tr}{\operatorname{Tr}}
\newcommand{\Tr}{\operatorname{Tr}}

\newcommand{\SO}{\operatorname{SO}}
\newcommand{\Orth}{\operatorname{O}}

\newcommand{\Mat}{\operatorname{M}}

\newcommand{\rank}{\operatorname{rank}}

\newcommand{\bbP}{\mathbb{P}}
\newcommand{\bbA}{\mathbb{A}}

\begin{document}

\title[Determinantal hypersurfaces]
{On the dimension of the locus of\\determinantal 
hypersurfaces}

\author[Reichstein]{Zinovy Reichstein$^\dagger$}

\author[Vistoli]{Angelo~Vistoli$^\ddagger$}

\address[Reichstein]{Department of Mathematics\\ University of British
Columbia \\ Vancouver, B.C., Canada V6T 1Z2}
\email{reichst@math.ubc.ca}

\address[Vistoli]{Scuola Normale Superiore\\ Piazza dei Cavalieri 7\\ 56126
Pisa\\ Italy}
\email{angelo.vistoli@sns.it}

\begin{abstract} The characteristic polynomial $P_A(x_0, \dots, x_r)$
of an $r$-tuple $A := (A_1, \dots, A_r)$ of $n \times n$-matrices is 
defined as
\[ P_A(x_0, \dots, x_r) := \det(x_0 I + x_1 A_1 + \dots + x_r A_r) \, . \]
We show that if $r \geqslant 3$
and $A := (A_1, \dots, A_r)$ is an $r$-tuple of $n \times n$-matrices 
in general position,
then up to conjugacy, there are only finitely many $r$-tuples 
 $A' := (A_1', \dots, A_r')$ such that $p_A = p_{A'}$. Equivalently,
the locus of determinantal hypersurfaces of degree $n$ in $\bP^r$
is irreducible of dimension $(r-1)n^2 + 1$.
\end{abstract}




\thanks{$^\dagger$Supported in part by an NSERC Discovery grant}
\thanks{$^\ddagger$Supported in part by research funds from 
the Scuola Normale Superiore}


\maketitle

\setcounter{tocdepth}{1} 

\section{Introduction}

Let $r, n \geqslant 2$ be integers, and $k$ be a base field.
Assume $\Char(k) = 0$ or $> n$.  Given an 
$r$-tuple $A := (A_1, \dots, A_r) \in \Mat_n^r$ 
of $n \times n$-matrices, we define {\em the characteristic polynomial}
of $A$ as
\[ P_A(x_0, \dots, x_r) := \det(x_0 I + x_1 A_1 + \dots + x_r A_r) \, , \]
where $I$ denotes that $n \times n$ identity matrix.
The purpose of this paper is to answer the following question,
due to B.~Reichstein. 

\begin{question} \label{q1}
For $(A_1, \dots, A_r)$ in general position in
$\Mat_n^r$, are there finitely many or infinitely many
conjugacy classes of $r$-tuples 
$A' := (A_1', \dots, A_r')$ such that $p_A = p_{A'}$?
\end{question}

To restate this question in geometric terms, consider the following 
diagram
\begin{equation} \label{e.main-diagram}
\xymatrix{ 
\Mat_n^r  \ar[dr]^P \ar[d]^{\pi} & &   \\
   Q_{r, n}= \Mat_n/\! \!/\PGL_n \ar[r]^-{\quad \overline{P} \quad} & \DHyp_{r, n} \ar@{^{(}->}[r] & 
 \Hypersurf_{r, n} \, .}  
\end{equation}
Here 
\begin{itemize}
\item $\Hypersurf_{r, n} \simeq \bbP^{\binom{r + n}{n} - 1}$ 
denotes the space of degree $n$
hypersurfaces in $\bbP^r$.  

\item
$Q_{r, n} := \Mat_n^r/\! \!/\PGL_n = \Spec k[\Mat_n^r]^{\PGL_n}$ denotes the 
categorical quotient space for the conjugation action of $\PGL_n$ on
$r$-tuples of $n \times n$-matrices.

\item
$\pi$ denotes the natural projection induced by the inclusion 
$k[\Mat_n^r]^{\PGL_n} \hookrightarrow k[\Mat_n^r]$.

\item
$P$ takes an $r$-tuple $A = (A_1, \dots, A_r)$ of $n \times n$ matrices
to the hypersurface in $\bbP^r$ cut out 
by the homogeneous polynomial $P_A(x_0, \dots, x_r)$ of degree $n$. 
Hypersurfaces of this form are called ``determinantal".

\item
$\DHyp_{r, n}$ denotes the closure of the image of $P$ in $\Hypersurf_{r, n}$
This is the ``locus of determinantal hypersurfaces" 
of degree $n$ in $\bbP^r$.
\end{itemize}

\begin{question} \label{q2} What is the dimension of $\DHyp_{r, n}$?
\end{question}

Questions~\ref{q1} and~\ref{q2} are closely related. Indeed, 
Question~\ref{q1} asks whether or not fibers of $\overline{P}$
in general position are finite, or equivalently, whether or not 
\[ \dim(\DHyp_{r, n}) = \dim(Q_{r, n}) \, ,  \] 
where
\[ \dim(Q_{r, n}) = \dim(\Mat_n^r) - \dim(\PGL_n) = (r-1)n^2 + 1 \, . \]
Our main result answers Questions~\ref{q1} and~\ref{q2} for $r \geqslant 3$.

\begin{theorem} 
\label{thm.main}
Assume $r \geqslant 3$. Then the map $\overline{P}$ is generically 
finite and separable. In particular,  
$\dim(\DHyp_{r, n}) = (r-1)n^2 + 1$, for any $n \geqslant 2$.
\end{theorem}

Several remarks are in order.

\smallskip
(1) A classical theorem of G.~Frobenius~\cite[$\S$7.1]{frobenius} 
asserts 
that the only linear transformations $T \colon \Mat_n \to \Mat_n$
preserving the determinant function are of the form $A \to P X Q$ or
$A \to P X^t Q$, where $X^t$ denotes the transpose of $X$, and $P$ and $Q$
are fixed $n \times n$ matrices, such that $\det(P) \det(Q) = 1$. 
(For modern proofs of this theorem, further references, 
and generalizations, see~\cite{dieudonne}, \cite[Theorem 2]{mm}, 
\cite[Theorem 4.2]{waterhouse}, \cite[Corollary 8.9]{bgl}.)
In the case where $r = n^2 -1$, Frobenius's theorem tells us that
the fiber of $\overline{P}$ contains exactly two points corresponding to
the conjugacy classes of $(A_1, \dots, A_r)$ and $(A_1^t, \dots, A_r^t)$, 
where $A^t$ denotes the transpose of $A$; see Lemma~\ref{lem.frobenius}.
In Section~\ref{sect.frobenius}
we will show that the same is true for any $r \geqslant n^2 - 1$.

\smallskip
(2) In the case where $n= r= 3$, Theorem~\ref{thm.main} 
is equivalent to the following assertion:
a general hypersurface of degree $3$ in
$\bbP^3$ is determinantal. Equivalently, the map  
$P \colon \Mat_3^3 \to \Hypersurf_{3, 3} \simeq \bbP^{19}$ is 
dominant. This result goes back to (at least) 
H.~Grassmann~\cite{grassmann}; for a modern 
proof (in arbitrary characteristic), 
see~\cite[Corollary 6.4]{beauville}.  

\smallskip
(3) In the case, where $r = 3$ and $n = 4$, Theorem~\ref{thm.main} 
is equivalent to the assertion of that determinantal quartic 
hypersurfaces in $\bbP^3$ form a codimension $1$ 
locus in $\Hypersurf_{3, 4} \simeq \bbP^{34}$. 
Over the field of complex numbers this is proved 
in~\cite[Example 4.2.23]{dolgachev}. 

\smallskip
(4) We do not know what the degree of $\overline{P}$ is in general; our proof 
of Theorem~\ref{thm.main} sheds no light on this question. As we mentioned
above, if $r \geqslant n^2 -1$, the general fiber of $\overline{P}$ 
consists of exactly two points corresponding to
the conjugacy classes of $(A_1, \dots, A_r)$ and $(A_1^t, \dots, A_r^t)$ 
(see Theorem~\ref{thm.frobenius}) and thus $\deg(\overline{P}) = 2$.
An interesting (and to the best of our knowledge, open)
question is whether or not
$\deg(\overline{P}) = 2$ for every $n \geqslant 2$ and $r \geqslant 4$.
Note however, that this fails for $r = 3$. Indeed,
if $r = n = 3$, then $\deg(\overline{P}) = 72$;
see~\cite{grassmann},~\cite[Corollary 6.4]{beauville} 
or~\cite[Theorem 9.3.6]{dolgachev}. 

\smallskip
(5) Theorem~\ref{thm.main} fails for $r = 2$, as long as
$n \geqslant 3$. Indeed, in this case
\[ \dim(Q_{2, n}) = n^2 + 1 >
\binom{n + 2}{2} - 1 = \dim(\Hypersurf_{2, n}) , \]
so the fibers of $\overline{P}$ cannot be finite.
In fact, this setting has been much studied, 
both from the theoretical point of view and in connection to
applications to control theory.  In particular, it is well 
known that the map 
$\overline{P} \colon Q_{2, n} \to \Hypersurf_{2, n}$ 
is dominant, and the points 
of the fiber of $\overline{P}$ over 
a general plane curve $C$ of degree $n$ are in a natural
bijective  correspondence with line bundles of degree
$\dfrac{n(n-1)}{2}$ on $C$.  For details
and further references, see~\cite{ct}, \cite{vinnikov}, 
\cite[Section 3]{beauville}, \cite[Section 4.1]{dolgachev}, \cite{neretin}.

\smallskip
(6) On the other hand, Theorem~\ref{thm.main} remains true for $r = n = 2$. 
Indeed, in this case the $k$-algebra $k[Q_{2, n}] = k[\Mat_2^2]^{\PGL_n}$
is generated by five algebraically independent elements, 
$\tr(A_1)$, $\tr(A_2)$, 
$\det(A_1)$, $\det(A_2)$ and $\tr(A_1 A_2)$; 
see,~\cite[Theorem 2.1]{procesi},
\cite[p. 20]{herstein} or \cite[Lemma 1(1)]{fhl}. 
One easily checks that these five elements lie in the 
$k$-algebra generated by the coefficients
of $\det(x_0 I + x_1 A_1 + x_2 A_2)$. We conclude that for $r = n = 2$
the map $\overline{P} \colon \Mat_2^2/\! \!/\PGL_2 \to \Hypersurf_{2, 2} 
\simeq \bbP^5$ is, in fact, a birational isomorphism, i.e.,
$\deg(\overline{P}) = 1$. If $r, n \geqslant 2$ but $(n, r) \neq (2, 2)$, 
then $(A_1, \dots, A_r)$ and $(A_1^t, \dots, A_r^t)$ are not conjugate, for
$(A_1, \dots, A_r) \in \Mat_n^r$ in general position (see, e.g., 
\cite[Remark 1 on p.~73]{reichstein}) and hence,
$\deg(\overline{P}) \geqslant 2$.

\smallskip
(7) The fact that $\overline{P} \colon \Mat_n^r \to \Hypersurf_{r, n}$ 
is dominant if and only if $r = 2$ or $r = n = 3$ was known
to L.~E.~Dickson; see~\cite{dickson}. Dickson also noted that
the determinantal form 
\[ \det(A_0 x_0 + \dots + A_r x_r)
 \sum_{i_0 + \dots + i_r = n} a_{i_0, \dots, i_r} 
x_0^{i_0} \dots x_r^{a_r} \, , \]
``involves no more than $(r-1) n^2 + 2$ parameters", i.e.,
the transcendence degree of the field generated by the coefficients
$a_{i_1}, \dots, a_{i_r}$ over $k$ is $\leqslant (r-1)n^2 + 2$;
see~\cite[Theorem 6]{dickson}.  
Our Theorem~\ref{thm.main} implies that this bound is, in fact, 
attained for the generic determinantal form. \footnote{The reason
for the discrepancy between $(r-1)n^2 + 2$ in Dickson's Theorem 6 and
$(r-1)n^2 + 1$ in our Theorem~\ref{thm.main} is that we take
$A_0 = I$. The ``extra" parameter in Dickson's setting is $\det(A_0)$.}

\smallskip
Our standing assumption on the base field $k$ is that $\Char(k) = 0$ or 
$> n$. Among other things, this allows us to use Newton's 
formulas to express the coefficients of the characteristic polynomial 
of an $n \times n$-matrix $X$ in terms of $\tr(X), \tr(X^2), \ldots, \tr(X^n)$.
Our main results are of a geometric nature, in the sense that in 
the course of proving them we may replace $k$ by a larger field. 
In particular, we may usually assume without loss of generality 
that $k$ is algebraically closed.
We do not know to what extent Theorem~\ref{thm.main} remains valid in
the case where $0 < \Char(k) \leqslant n$; our argument breaks 
down in this setting.

A modern approach to the study of determinantal hypersurfaces 
is based on the fact that a hypersurface $X \subset \bbP^n$ 
is determinantal if and only if $X$ carries an Ulrich sheaf of rank $1$; 
see~\cite{beauville} in the case, where $X$ is smooth,
and \cite{es} in general. We have not been able to prove
Theorem~\ref{thm.main} using this approach, even though 
this may well be possible (one complication is that 
for $r > 3$ every determinantal hypersurface 
is singular). The proof we give here is entirely elementary.

\subsection*{Acknowledgments} We would like to thank Boris 
Reichstein for bringing Question~\ref{q1} to our attention. 
We are also grateful to 
Arnaud Beauville for helpful comments, and to the referees 
for calling our attention 
to~\cite{es} and encouraging us to include a proof
of Theorem~\ref{thm.frobenius} in this paper.
We are in debt to Marian Aprodu for his interest 
in our work and his very useful observations.

\section{A general strategy for the proof of Theorem~\ref{thm.main}}
\label{sect.strategy}

The first step is to reduce Theorem~\ref{thm.main}
to the case where $r = 3$. We will do this in Section 3, then
assume that $r = 3$ for the rest of the proof.  Clearly 
\begin{equation} \label{e.dim>}
\dim(\DHyp_{3, n}) \leqslant \dim(Q_{3, n}) = 2n^2 + 1,
\end{equation}
since the morphism $\overline{P} \colon Q_{3, n} \to \DHyp_{3, n}$ is
dominant, by definition. The following lemma will supply a key ingredient 
for our proof of Theorem~\ref{thm.main}.

\begin{lemma} \label{lem.key}
There exists a triple of $n \times n$ matrices $A = (A_1, A_2, A_3)
\in \Mat_n^3$ such that the differential $dP_{|A}$ of $P$ at $A$ has
rank $2n^2 + 1$. 
\end{lemma}

Once Lemma~\ref{lem.key} is established, we know that 
$\rank \, dP_{|B} \geqslant 2n^2 + 1$ 
for $B \in \Mat_n^3$ is general position. Hence,~\eqref{e.dim>} 
is an equality.  Moreover, for $B \in \Mat_n^3$ 
in general position 
\[ \rank \, d \overline{P}_{|\pi(B)} \geqslant
\rank \, d P_{| B} = 2n^2 + 1 \, .  \]
Since $\dim(Q_{3, n}) = \dim(\DHyp_{3, n}) = 2n^2 + 1$,
we conclude that for $B \in \Mat_3^r$ in general position,
$d\overline{P}_{|\pi(B)}$ is an isomorphism. In other words,
$\overline{P}$ is generically finite and separable, as desired.

Our proof of Lemma~\ref{lem.key} will be structured as follows.
In Section~\ref{sect.tangent-space} we will exhibit a homogeneous system 
of linear equations cutting out $\Ker(dP_{|A})$ 
inside the tangent space $T_A(\Mat_n^3)$
(which we identify with $\Mat_n^3$) in Section~\ref{sect.tangent-space}.
We will do this for any triple $A = (A_1, A_2, A_3) \in \Mat_n^3$
such that the linear span of $A_1$, $A_2$ and $A_3$ in $\Mat_n$
contains a matrix with distinct eigenvalues; 
see Lemma~\ref{lem.differential}(b).
Our goal will be to prove Lemma~\ref{lem.key} by showing that
$\dim \, \Ker(dP_{|A}) = n^2 - 1$.
The system of linear equations we obtain, cutting out $\Ker(dP_{|A})$
in $\Mat_n^3$, is rather complicated (in particular, it is badly 
overdetermined). For this reason we have not been able to compute 
the dimension of $\Ker(dP_{|A})$ for an arbitrary triple 
$A = (A_1, A_2, A_3) \in M_n^3$ whose linear span contains
a matrix with distinct eigenvalues.
However, for the particular triple $A = (A_1, A_2, A_3)$ 
defined in~\eqref{e.A1-3}, the kernel of $dP_{|A}$ 
carries a $(\bbZ/n \bbZ)^2$-grading, i.e., 
remains invariant under a certain linear action of 
the finite abelian group $G := (\bbZ/n \bbZ)^2$ on $M_n^3$; see
Section~\ref{sect.grading}.  This will allow us to 
decompose $\Mat_n^3$ as a direct sum of $n^2$ three-dimensional 
character spaces, and verify that $\Ker(dP_{|A})$ has the desired dimension,
$n^2 - 1$, by solving our linear system in each character space. 
This computation, completing the proof of Lemma~\ref{lem.key}
(and thus of Theorem~\ref{thm.main}), will be 
carried out in Sections~\ref{sect.grading} and \ref{sect.final}. 
 It relies on properties of 
$q$-binomial and trinomial coefficients, which are recalled 
in~Section~\ref{sect.skew}.

\section{Reduction to the case, where $r = 3$}
\label{sect.r=3}

Throughout this section, we will fix $n \geqslant 2$ and
denote the map
\[ \Mat_n^r/\! \!/\PGL_n \to \DHyp_{r, n} \]
in diagram~\eqref{e.main-diagram} by $\overline{P}(r, n)$. 

\begin{proposition} \label{prop.r+} Assume $r \geqslant 3$. 
If the morphism
$\overline{P}(r, n)$ is generically finite and separable, 
then so is $\overline{P}(r + 1, n)$.
\end{proposition}

Let $K_{r, n}:= k(\Mat_n^r)^{\PGL_n}$ be the field of rational functions
on $\Mat_n^r/\! \!/\PGL_n$ and $K_{r, n}'$ be the subfield of $K_{r, n}$
generated by the coefficients of the characteristic polynomial
\[ (A_1, \dots, A_r) \quad \mapsto \quad 
\det(x_0 I + x_1 A_1 + \dots + x_r A_r) \, . \]
Clearly $K_{r, n}'$ is the field of rational functions on
$\DHyp_{r, n}$ and the inclusion of function fields 
$P^* \colon k(\DHyp_{r, n}) \hookrightarrow k(Q_{r, n})$ is 
the natural inclusion $K_{r, n}' \hookrightarrow K_{r, n}$.
Thus Proposition~\ref{prop.r+} can be restated, in purely 
algebraic terms, as follows.

\begin{proposition} \label{prop.r++} 
Assume $r \geqslant 3$. If the field extension $K_{r, n}/K_{r, n}'$ is 
finite and separable, then so is
$K_{r + 1, n}/K_{r+1, n}'$.  
\end{proposition}

The key to our proof of Proposition~\ref{prop.r++} is the following
lemma which asserts that $K_{r, n}$ is generated, as a field 
extension of $k$, by
functions that depend on at most three of the matrices $A_1, \dots, A_r$.

\begin{lemma} {\rm(}C.~Procesi{\rm)} \label{lem.procesi} 
Assume $r \geqslant 3$.  There are finitely many monomials
$M_1, \ldots, M_N$ in $A_1$ and $A_2$ such that $K_{r, n}$ is generated, 
as a field extension of $k$, by the elements
$\tr(M_i)$ and $\tr(M_i A_j)$, where $i = 1, \dots, N$, and $j = 3, \dots, r$.
\end{lemma}

\begin{proof} See~\cite[Proposition 2.3, p.~255]{procesi} 
or \cite[Theorem 3.2 and Example 3.3(a)]{fgg}.
\end{proof}

\begin{proof}[Proof of Proposition~\ref{prop.r++}]
First observe that $K_{r, n} \subset K_{r+1, n}$ and 
$K_{r, n}' \subset K_{r+1, n}'$ (just set $A_{r+1} = 0$).

By Lemma~\ref{lem.procesi}, there exist finitely many monomials
$M_1, \ldots, M_N$ in $A_1$ and $A_2$ such that $K_{r+1, n}$ is generated,
as a field extension of $k$, by $\tr(M_i)$ and $\tr(M_i A_j)$, 
where $i = 1, \dots, N$, and $j = 3, \dots, r + 1$.
It thus suffices to show that each of these elements
is algebraic and separable over $K_{r+1, n}'$. 

Let us start with $\tr(M_i)$. By definition, $\tr(M_i) \subset K_{2, n} 
\subset K_{r, n}$.  By our assumption 
$\Tr(M_i)$ is thus algebraic and separable over $K_{r, n}'$. Since 
$K_{r, n}' \subset K_{r+1, n}'$, $\Tr(M_i)$ is algebraic and separable
over $K_{r+1, n}'$, as desired.

Similarly $\tr(M_iA_3) \subset K_{3, n} \subset K_{r, n}$, since 
$r \geqslant 3$.
By our assumption $\tr(M_i A_3)$ is algebraic and separable
over $K_{r, n}'$. Hence, it is algebraic and separable over $K_{r+1, n}'$.
By symmetry $\tr(M_i A_j)$ is also algebraic and separable over 
$K_{r+1, n}'$ for every $j = 3, \dots, r+1$, and the proof of
Proposition~\ref{prop.r++} is complete.
\end{proof}

\section{The kernel of $dP$}
\label{sect.tangent-space}

Observe that the image of the map $P$ lies in the affine subspace 
$\bbA^{\binom{r + n}{n} - 1}$ of $\bbP^{\binom{r + n}{n} - 1} = 
\Hypersurf_{r, n}$ consisting of hypersurfaces of 
the form 
\[ \sum_{i_0 + \dots + i_r = n} 
a_{i_1, \ldots, i_r} x_0^{i_0} \dots x_r^{i_r} = 0 \, ,  \]
where $a_{n, 0, \ldots, 0} \neq 0$ (or equivalently,
$a_{n, 0, \ldots, 0} = 1$, after rescaling). Thus we may view
$P$ as a polynomial map between the affine spaces $\Mat_n^r$ and 
$\bbA^{\binom{r + n}{n} - 1}$. The differential $dP_{|A}$ 
at a point $A \in \Mat_n^r$ is a linear map 
$T_A(\Mat_n^r) \to T_A(\bbA^{\binom{r + n}{n} - 1})$. We will identify
$T_A(\Mat_n^r)$ with $\Mat_n^r$ and $T_A(\bbA^{\binom{r + n}{n} - 1})$ with
$\bbA^{\binom{r + n}{n} - 1}$ in the obvious way.

Given an $n \times n$ matrix $X$, we will denote the classical
adjoint of $X$ by $X^{ad}$. Recall that $X^{ad}$ is, by definition,
the $n \times n$ matrix whose $(i,j)$-component is
$(-1)^{i + j} \det(X_{ji})$, where $X_{ji}$
is the $(n-1) \times (n-1)$ matrix obtained from $X$ by deleting 
row $j$ and column $i$.  If $X$ is invertible, then 
$X^{ad} = \det(X) X^{-1}$. 

\begin{lemma} \label{lem.differential}
Let $A = (A_1, \dots, A_r)$ be an $r$-tuple of $n \times n$-matrices.

\smallskip
(a) The differential $dP_{|A}$ sends 
$(B_1, \dots, B_r) \in T_A(\Mat_n^r) \simeq \Mat_n^r$ 
to 
\[ \tr( (x_0 I + x_1 A_1 + \dots + x_r A_r)^{ad} 
(x_1 B_1 + \dots + x_r B_r) ). \]

\smallskip
(b) Suppose some matrix in the linear span of 
$A_1, \dots, A_r$ has distinct eigenvalues. Then
the kernel of $dP_{|A}$ is the space of 
$r$-tuples $(B_1, \dots, B_r) \in \Mat_n^r$ satisfying
\[ \tr( (x_1 A_1 + \dots + x_r A_r)^d 
(x_1 B_1 + \dots + x_r B_r) )= 0 \]
for every $d = 0, 1, \dots, n - 1$. 
\end{lemma}

In part (b) we require that for every $d = 0, 1, \dots, n-1$ 
the left hand side of the 
formula should be identically zero as a polynomial in $x_1, \dots, x_r$.  
This gives rise to a system of linear equations
in $(B_1, \dots, B_r) \in \Mat_n^r$, whose solution space is $\Ker(dP_{|A})$.

\begin{proof} (a) 
Let $Y = (y_{ij})$ and  $\Delta Y = (\Delta y_{ij})$ be 
$n \times n$ matrices. We think of
the entries $\Delta y_{ij}$ as being ``small"
and of the entries of $Y$ as being constant.
We claim that
\begin{equation} \label{e2.3}
\det(Y + \Delta Y) = \det(Y) + \tr(Y^{ad} \Delta Y) + 
(\text{terms of degree $\geqslant 2$ in 
$\Delta y_{ij}$}). 
\end{equation}
The case where $Y = I$ is easy: the usual expansion of 
the characteristic polynomial of $\Delta Y$, yields
\begin{equation} \label{e2.2}
\det(I + \Delta Y) = 1 + \tr(\Delta Y) + (\text{terms of degree $\geqslant 2$
in $\Delta y_{ij}$}). 
\end{equation}
To prove the claim for arbitrary $Y$, note that
both sides of~\eqref{e2.3} are $n \times n$-matrices, whose entries are 
polynomials in $y_{ij}$ and $\Delta y_{ij}$. Hence, 
in order to establish~\eqref{e2.3} for an arbitrary $Y$,  
we may assume without loss of generality that $Y$ is non-singular.
In this case, 
\[ \det(Y + \Delta Y) = \det(Y) \det(I + Y^{-1} \Delta Y) \, .  \]
Expanding the second factor as in~\eqref{e2.2}, we arrive at~\eqref{e2.3}.
This completes the proof of the claim.

In order to finish the proof of part (a), we will 
compute the directional derivative of $P$
in the direction of $(B_1, \dots, B_r) \in \Mat_n^r$. 
Setting $Y:= x_0 I + x_1 A_1 + \dots + x_r A_r$ and
$\Delta Y := (x_1 B_1 + \dots + x_r B_r) h$, and 
applying~\eqref{e2.3}, we see that
\begin{align*} 
P(A_1 + hB_1, \dots A_r + hB_r) = \det(Y + \Delta Y) = 
\det(Y + \Delta Y) = \det(Y) + \tr(Y^{ad} \Delta Y) h + O(h^2) \\ 
= 
P(A_1, \dots, A_r) + \tr( (x_0 I + x_1 A_1 + \dots + x_r A_r)^{ad} 
(x_1 B_1 + \dots + x_r B_r) ) h + O(h^2) \, . 
\end{align*} 
This shows that the directional derivative
of $P$ at $A$ in the direction of $B$ 
is \[ \tr( (x_0 I + x_1 A_1 + \dots + x_r A_r)^{ad} 
(x_1 B_1 + \dots + x_r B_r) ), \] and part (a) follows. 
(Note that in the last computation
$h \to 0$ but $x_0, x_1, \dots, x_n$ remain constant throughout.) 

\smallskip
(b) Let $A$ be an $n \times n$ matrix 
with distinct eigenvalues, over a field $K$.
We claim that $B \in \Mat_n$ satisfies 

\smallskip
(i) $\tr((x_0 I + A)^{ad} B) = 0$ for every
$x_0$ 

\smallskip
\noindent
if and only if $B$ satisfies 

\smallskip
(ii) $\tr(A^d B ) = 0$ 
for every $d = 0, \dots, n-1$.

\smallskip
Once this claim is established, we can deduce part (b) from part (a)
by setting $A := x_1 A_1 + \dots + x_r A_r$ and 
$B := x_1 B_1 + \dots + x_r B_r$ and working over the field 
$K = k(x_1, \dots, x_r)$.

To prove the claim, we may pass to the algebraic closure of $K$.
By our assumption $A$ has distinct eigenvalues, and hence, 
is diagonalizable.
We may thus assume without loss of generality that 
$A$ is the diagonal matrix $\diag(\lambda_1, \dots, \lambda_n)$,
where $\lambda_1, \dots, \lambda_n$ are distinct elements of $K$.
Then 
\[ (tI + A)^{ad} = \diag(\dfrac{\Pi(t)}{t + \lambda_1}, \dots, 
\dfrac{\Pi(t)}{t + \lambda_n}), \]
where $\Pi(t) = (t + \lambda_1) (t + \lambda_2) \dots (t + \lambda_n) = 
\det(tI + A)$ and each diagonal entry 
$\dfrac{\Pi(t)}{t + \lambda_i}$ is a polynomial of degree $n-1$ in $t$.
Condition (i) now translates to
\[ \sum_{i = 1}^n b_{ii} \dfrac{\Pi(t)}{t + \lambda_i} = 0 , \]
where $b_{11}, \dots, b_{nn}$ are the diagonal entries of $B$.
Setting $t = - \lambda_i$, for $i = 1, \dots, n$,
we obtain $b_{11} = b_{22} = \dots = b_{nn} = 0$.
On the other hand, condition (ii) translates to
\[ \sum_{i = 1}^n \lambda_i^d b_{ii} = 0 , \]
for each $d = 0, 1, \dots, n-1$, which we view as a homogeneous
system of $n$ linear equations in $n$ unknowns $b_{11}, \dots, b_{nn}$. 
The matrix of this system is the Vandermonde matrix
\[ \begin{pmatrix} 1 & 1 & \dots & 1 \\
\lambda_1 & \lambda_2 & \dots & \lambda_n \\
\vdots & \vdots & \vdots & \vdots  \\
\lambda_1^{n-1} & \lambda_2^{n-1} & \dots & \lambda_n^{n-1} 
\end{pmatrix} . \]
Since $\lambda_1, \dots, \lambda_n$ are distinct, this 
Vandermonde matrix
is non-singular, and the above system has only the trivial solution,
$b_{11} = b_{22} = \dots = b_{nn} = 0$. 

In summary, for 
$A = \diag(\lambda_1, \dots, \lambda_n)$ both (i) and (ii) are equivalent
to $b_{11} = b_{22} = \dots = b_{nn} = 0$. Hence, (i) and (ii) are equivalent
to each other. This completes the proof of the claim and thus of 
Lemma~\ref{lem.differential}(b).
\end{proof}

\section{Skew-commuting matrices and $q$-binomial coefficients}
\label{sect.skew}

Recall that we are working over a base field $k$
of characteristic $0$ or $> n$. For the sake of proving
Theorem~\ref{thm.main}, we may assume without loss of generality that
$k$ is algebraically closed.  In particular, we may assume that
$k$ contains a primitive $n$th root of unity, which we will denote 
by $q$. We will also assume that $r = 3$; see Proposition~\ref{prop.r+}(a).  
For the remainder of the proof of Theorem~\ref{thm.main}, we will set
\begin{equation} \label{e.A1-3}
A_1 := \begin{pmatrix} 1 & 0 & 0 & \dots & 0 \\
                          0 & q & 0 & \dots & 0 \\
                          \hdotsfor{5} \\
                          0 & 0 & 0 & \dots & q^{n-1}
\end{pmatrix},
\; \; 
A_2 := \begin{pmatrix} 0 & 1 & 0 & \dots & 0 \\
 0 & 0 & 1 & \dots & 0 \\
                          \hdotsfor{5} \\
                          1 & 0 & 0 & \dots & 1
\end{pmatrix},
\; \;  \text{and} \; \; A_3 := A_1 A_2.
\end{equation}
It is easy to see that 
\[ A_2 A_1 = q A_1 A_2 \, , \quad \text{and} \quad A_1^n = A_2^n = I \, , \]
where, as usual, $I$ denotes that $n \times n$-identity matrix.
Hence, conjugation by $A_1$ commutes with conjugation by $A_2$; we will denote
these commuting linear operators by $\Conj_{A_1}$ and $\Conj_{A_2} \colon
\Mat_n \to \Mat_n$, respectively. They generate a subgroup of
$\GL(\Mat_n)$ isomorphic to $(\bbZ/n \bbZ)^2$.  
One readily checks that
\[ \text{$\Conj_{A_1} ( A_1^{e_1} A_2^{e_2}) = q^{-e_2} A_1^{e_1} A_2^{e_2}$
and 
$\Conj_{A_2} ( A_1^{e_1} A_2^{e_2}) = q^{e_1} A_1^{e_1} A_2^{e_2}$.} \]
In particular,
\begin{equation} \label{e.trace}
\tr ( A_1^{e_1} A_2^{e_2}) = \begin{cases} 
\text{$n$, if $e_1 \equiv e_2 \equiv 0 \pmod{n}$, and} \\
\text{$0$, otherwise}.
\end{cases}
\end{equation}
Letting $e_1$ and $e_2$ range over $\bbZ/n \bbZ$, we see that
each of the $n^2$ one-dimensional subspaces
$\Span_k(A_1^{e_1} A_2^{e_2})$ is a character space for 
the abelian group
\[ \langle \Conj_{A_1} , \Conj_{A_2} \rangle \simeq
 (\bbZ/n \bbZ)^2 \, . \] 
Since these spaces have distinct associated characters,
the matrices $A_1^{e_1} A_2^{e_2}$ form a $k$-basis
of $\Mat_n$, as $e_1$ and $e_2$ range over $\bbZ/n \bbZ$.
In the sequel it will often be more convenient for us to
work in this basis than in
the standard basis of $\Mat_n$, consisting of elementary matrices.

We now recall that the $q$-{\em factorial} $[d]_q!$ of 
an integer $d \geqslant 0$ is given by
\[ [d]_q ! := [1]_q [2]_q \dots [d]_q \, , \]
where $[a]_q : = \dfrac{1 - q^a}{1 - q} = 1 + q + \dots + q^{a-1}$.
In particular, $[0]_q ! = 1$. (Recall that we are assuming that 
$n \geqslant 2$ throughout, and thus $q \neq 1$.)
If $a$ and $b$ are non-negative integers and $a + b = d \leqslant n-1$, then 
\begin{equation} \label{e.q-binomial}
\binom{d}{a, b}_q : = \frac{[d]_q !}{[a]_q ! [b]_q !} \, . 
\end{equation}
is called a $q$-{\em binomial coefficient}.  If $a < 0$ or $b < 0$, we set
\[ \binom{d}{a, b}_q : = 0 \, . 
\]
Similarly, if $a + b + c = d \leqslant n-1$, then
\begin{equation} \label{e.q-trinomial}
\binom{d}{a, b, c}_q : = \begin{cases}
\text{{$\dfrac{[d]_q !}{[a]_q ! \,  [b]_q ! \,  [c]_q !}$}, 
if $a, b, c \geqslant 0$, and} \\
 \\
\text{$0$, otherwise.}
\end{cases} 
\end{equation}
is called a $q$-{\em trinomial coefficient}. This terminology is justified by
parts (a) and (b) of the following lemma. Part (c) will play an important role
in the sequel.

\begin{lemma} \label{lem.binomial} Assume $d = 0, \dots, n-1$. 

\smallskip
(a) Let $X$ and $Y$ be matrices such that $XY = q YX$. Then
\[ (X + Y)^d = \sum_{a + b = d} 
\binom{d}{a, b}_q X^a Y^b  \, . \]

\smallskip
(b) Let $A_1$ and $A_2$ be as in~\eqref{e.A1-3}. Then 
\[ (x_1 A_1 + x_2 A_2 + x_3 A_1 A_2)^d = 
\sum_{ a + b + c = d}
q^{\frac{c(c-1)}{2}} 
\binom{d}{a, b, c}_q 
x_1^a x_2^b x_3^c 
A_1^{a + c} A_2^{b + c}. \] 

\smallskip
(c) For any $e_1, e_2 \in \bbZ/n \bbZ$, 
\[ \tr((x_1 A_1 + x_2 A_2 + x_3 A_1 A_2)^d A_1^{e_1} A_2^{e_2}) =
n \sum_{a, b, c}
q^{e_1(b+c) + \frac{c(c-1)}{2}} 
\binom{d}{a, b, c}_q 
x_1^a x_2^b x_3^c , \]
where the sum ranges over triples of non-negative integers $(a, b, c)$, subject
to the following conditions:
$a + b + c = d$, 
$a + c +  e_1 \equiv 0 \pmod{n}$, and 
$b + c + e_2 \equiv 0 \pmod{n}$.
\end{lemma}

\begin{proof} The binomial formula in part (a) was proved by
M.~P.~Sch\"utzenberger~\cite{schutzenberger}; for a 
detailed discussion of this formula and further references, 
see~\cite{monthly}.

(b) We apply part (a) twice. First we set 
$X = x_1 A_1 + x_3 A_1 A_2$ and $Y := x_2 A_2$ to obtain  
\begin{equation} \label{e4.1} (x_1 A_1 + x_2 A_2 + x_3 A_1 A_2)^d = 
\sum_{i + j = d}
\binom{d}{i, j}_q
(x_1 A_1 + x_3 A_1 A_2)^i x_2^{j} A_2^{j}.
\end{equation}
Next we apply part (a) with $X := x_1A_1$ and $Y := x_3 A_1 A_2$:
\begin{equation} \label{e4.2} (x_1 A_1 + x_3 A_1 A_2)^i = 
\sum_{ a + c = i} 
\binom{i}{a, c}_q 
x_1^a x_3^{c} A_1^a (A_1 A_2)^c . 
\end{equation} 
Substituting~\eqref{e4.2} into \eqref{e4.1}, setting $i := a + c$ and $b:= j$, 
and using the identities
\begin{equation} \label{e.binomial-formula}
\binom{d}{a, b, c}_q =  \binom{d}{i, b}_q 
\binom{i}{a, c}_q 
\end{equation}
and  
\begin{equation} \label{e.skew2}
(A_1 A_2)^c = q^{\frac{c(c-1)}{2}} A_1^c A_2^c \, , 
\end{equation}
we obtain the formula in part (b). Note that~\eqref{e.binomial-formula}
is an immediate consequence of 
the definitions~\eqref{e.q-binomial} and~\eqref{e.q-trinomial}, 
and~\eqref{e.skew2} follows from $A_2A_1 = q A_1 A_2$.

To deduce part (c) from part (b), multiply both sides of (b) 
by $A_1^{e_1} A_2^{e_2}$, rewrite 
$A_2^{b + c} A_1^{e_1}$ as $q^{e_1(b + c)} A_1^{e_1} A_2^{b + c}$,
and take the trace on both sides. The desired equality now follows 
from~\eqref{e.trace}.
\end{proof}

For future reference we record a simple identity involving $q$-trinomial 
coefficients.

\begin{lemma} \label{lem.ratios}
Suppose $\alpha$, $\beta$, and $\gamma$ are integers,
$0 \leqslant \alpha, \beta, \gamma \leqslant n-1$ and 
$1 \leqslant \alpha + \beta + \gamma \leqslant n$. 
Set $d := \alpha + \beta + \gamma - 1$. Then  
\[ (\binom{d}{\alpha-1, \beta, \gamma}_q : 
 \binom{d}{\alpha, \beta -1 , \gamma}_q : 
 \binom{d}{\alpha, \beta, \gamma -1}_q ) =
(1 - q^{\alpha}: 1 - q^{\beta}: 1 - q^{\gamma}) \] 
as points in the projective plane $\bbP^2$.
\end{lemma}

\begin{proof} If $\alpha, \beta, \gamma > 0$, the lemma is obtained 
by multiplying each of the numbers
\[ \binom{d}{\alpha-1, \beta, \gamma}_q, \quad   
\binom{d}{\alpha, \beta -1, \gamma}_q , \quad \text{and} \quad 
 \binom{d}{\alpha, \beta, \gamma-1}_q 
\]
by the non-zero scalar
$(1- q) \dfrac{[\alpha]_q ! \, [\beta]_q! \, [\gamma]_q!}{[d]_q !} \in k$.
If one of the integers $\alpha, \beta, \gamma$ is $0$, say, $\alpha = 0$,
then 
\[ \binom{d}{\alpha-1, \beta, \gamma}_q = 1 - q^{\alpha} = 0 \, , \]
and the lemma follows.   
\end{proof}

\section{A grading of $\Ker(dP_{|A})$}
\label{sect.grading}

Let $A_1$, $A_2$ and $A_3 = A_1 A_2$ be as in~\eqref{e.A1-3}.
Let $V := \Ker(dP_{|A}) \subset \Mat_{n}^3$, where the map 
$P \colon \Mat_n^3 \to \Hypersurf_{3, n}$ is defined in
the Introduction.  Since $A_1$ has distinct eigenvalues, 
Lemma~\ref{lem.differential}(b) tells us that $V \subset \Mat_n^3$ 
consists of triples $(B_1, B_2, B_3)$ satisfying
\[ \tr((x_1 A_1 + x_2 A_2 + x_3 A_1 A_2)^d (x_1 B_1 + x_2 B_2 + x_3 B_3)) = 0 \]
for $d = 0, 1, \dots, n-1$. Here the left hand side is required to be 
zero as a polynomial in $x_1, x_2, x_3$, for every  
$d = 0, 1, \dots, n-1$. 

Following the strategy outlined in Section~\ref{sect.strategy},
in order to complete the proof of Theorem~\ref{thm.main} 
(or equivalently, of Lemma~\ref{lem.key}), it suffices 
to show that $\dim(V) = n^2 - 1$.

\begin{lemma} \label{lem.invariance} 
$V$ is invariant under the linear action of the finite 
abelian group
$(\bbZ/n \bbZ)^2 = \langle \tau, \sigma \rangle $ on $\Mat_n^3$ given 
by 
\begin{gather} \label{e.grading}
\sigma \colon (B_1, B_2, B_3) \mapsto (\Conj_{A_1}(B_1), 
q \Conj_{A_1}(B_2), q \Conj_{A_1}(B_3)) \\
\tau \colon (B_1, B_2, B_3) \mapsto (q^{-1} \Conj_{A_2}(B_1), 
 \Conj_{A_2}(B_2), q^{-1} \Conj_{A_2} (B_3) )  \, .
\end{gather}
\end{lemma}

\begin{proof} Suppose $(B_1, B_2, B_3) \in V$, i.e., 
\[ f_{B_1, B_2, B_3, d}(x_1, x_2, x_3) := 
\tr((x_1 A_1 + x_2 A_2 + x_3 A_1 A_2)^d
(x_1 B_1 + x_2 B_2 + x_3 B_3)) = 0 \]
for every $d = 0, \dots, n-1$. Here $f_{B_1, B_2, B_3, d}$ 
is a polynomial in $x_1, x_2, x_3$ with coefficients in $k$,
and $f_{B_1, B_2, B_3, d}(x_1, x_2, x_3) = 0$ means 
that $f_{B_1, B_2, B_3, d}$ is the zero polynomial, i.e., 
every coefficient vanishes.  Let 
\[ (C_1, C_2, C_3) := \sigma(B_1, B_2, B_3) =
(\Conj_{A_1}( B_1)  , q \Conj_{A_1} (B_2), q \Conj_{A_1}(B_3)) \, ,\]
as above.  To prove that $V$ is invariant
under $\sigma$, we need to show that $(C_1, C_2, C_3) \in V$, i.e., 
$f_{C_1, C_2, C_3, d}$
is identically $0$ for every $d= 0, 1, \dots, n-1$.
Keeping in mind that 
\[ A_1 := \Conj{A_1} (A_1), \quad 
A_2 := q \Conj_{A_1} (A_2), \quad \text{and} \quad 
A_1 A_2 := q \Conj_{A_1}(A_1 A_2), \]
we see that
\begin{gather}
\nonumber
0 = f_{B_1, B_2, B_3, d}(x_1, x_2, x_3)  
=   \tr( \Conj_{A_1} ((x_1 A_1 + x_2 A_2 + x_3 A_1 A_2)^d
(x_1 B_1 + x_2 B_2 + x_3 B_3)) \\
\nonumber
 = \tr((x_1 A_1 + x_2 q^{-1} A_2 + x_3 q^{-1} A_1 A_2)^d
  (x_1 C_1 +  x_2 q^{-1} C_2 + 
x_3 q^{-1} C_3)) \\ 
\nonumber
= f_{C_1, C_2, C_3, d}(x_1, q^{-1} x_2, q^{-1} x_3).  
\end{gather}
This shows that 
$f_{C_1, C_2, C_3, d}(x_1, q^{-1} x_2, q^{-1} x_3)$ is identically zero
as a polynomial in $x_1, x_2, x_3$.  Hence, so is   
$f_{C_1, C_2, C_3, d}(x_1, x_2, x_3)$, as desired.

A similar argument shows that $V$ is invariant under $\tau$. (Here 
we conjugate by $A_2$, rather than $A_1$.)
This completes the proof of Lemma~\ref{lem.invariance}.
\end{proof}
  
Since we are working over an algebraically closed base field $k$ 
and $\Char(k) = 0$ or $> n$,
Lemma~\ref{lem.invariance} tells us that $V$ is a direct sum of 
character spaces for the action of 
$(\bbZ/ n \bbZ)^2$ on $\Mat_n^3$.
There are $n^2$ character spaces, each of dimension $3$
(one for each character of $(\bbZ/ n \bbZ)^2$). They are defined as follows
\[ W_{e_1, e_2} := \{ (t_1 A_1^{e_1 + 1} A_2^{e_2}, 
t_2 A_1^{e_1} A_2^{e_2 + 1}, t_3 A_1^{e_1 + 1} A_2^{e_2 + 1}) \, | \,
t_1, t_2, t_3 \in k \} , \]
where $(e_1, e_2) \in (\bbZ/n \bbZ)^2$. Here $\sigma$ multiplies
every vector in $W_{e_1, e_2}$ by $q^{-e_2}$ and $\tau$ by $q^{e_1}$.
In other words, $(\bbZ/n \bbZ)^2$ acts on $W_{e_1, e_2}$
by the character
\[ \chi \colon \sigma^a \tau^b \mapsto q^{-e_2 a + e_1 b} \, . \]
In summary, 
$V = \bigoplus_{e_1, e_2  = 0}^{n-1} V_{e_1, e_2}$, 
where \[ V_{e_1, e_2} := V \cap W_{e_1, e_2}. \] 
Recall that our goal is to show that $\dim(V) = n^2 - 1$. 
Thus in order to prove Theorem~\ref{thm.main}, it suffices 
to establish
the following proposition. 

\begin{proposition} \label{prop.final}

\smallskip
{\rm (a)} $V_{0, 0} = (0)$.

\smallskip
{\rm (b)} $\dim(V_{e_1, e_2}) = 1$ for any $(0, 0) \neq (e_1, e_2) \in
(\bbZ/ n \bbZ)^2$.  
\end{proposition}

Proposition~\ref{prop.final} will be proved in the next section.

\begin{remark} \label{rem.commutator} If $X$ and $Y$ are 
$n \times n$-matrices, then clearly
$\Tr(X^d[X, Y]) = 0$ for every $d \geqslant 0$.
Setting $X = x_1 A_1 + x_2 A_2 + x_3 A_1 A_2$, $Y = A_1^{e_1}A_2^{e_2}$,
and thus
\[ [X, Y] = x_1 (1 - q^{e_2}) A_1^{e_1 + 1} A_2 + x_2 (q^{e_1} - 1) 
A_1^{e_1} A_2^{e_2 + 1} + x_3 (q^{e_1} - q^{e_2}) A_1^{e_1 + 1} A_2^{e_2 + 1} 
\, , \]
we see that the triple 
\[ (B_1, B_2, B_3) = ((1 - q^{e_2}) A_1^{e_1 + 1} A_2^{e_2}, (q^{e_1} - 1) 
A_1^{e_1} A_2^{e_2 + 1}, (q^{e_1} - q^{e_2}) A_1^{e_1 + 1} A_2^{e_2 + 1}) 
\]
lies in $V_{e_1, e_2}$. Here $(B_1, B_2, B_3) = (0, 0, 0)$
if $(e_1, e_2) = (0, 0)$ in $(\bbZ/n \bbZ)^2$ and
$(B_1, B_2, B_3) \neq (0, 0, 0)$ otherwise.
Proposition~\ref{prop.final} tells us that, in fact,
$(B_1, B_2, B_3)$ spans $V_{e_1, e_2}$ for every $(e_1, e_2) \in 
(\bbZ/n \bbZ)^2$.
\end{remark}

\section{Conclusion of the proof of Theorem~\ref{thm.main}}
\label{sect.final}

It remains to prove Proposition~\ref{prop.final}.
Given $t_1, t_2, t_3 \in k$, recall that an element 
\[ w := (t_1 A_1^{e_1 + 1} A_2^{e_2}, 
t_2 A_1^{e_1} A_2^{e_2 + 1}, t_3 A_1^{e_1 + 1} A_2^{e_2 + 1}) \] 
of $W_{e_1, e_2}$ lies in $V_{e_1, e_2}$ if and only if
\[ \tr((x_1 A_1 + x_2 A_2 + x_3 A_1 A_2)^d(t_1 x_1 A_1^{e_1+1} A_2^{e_2} +
t_2 x_2 A_1^{e_1} A_2^{e_2 + 1} + t_3 x_3 A_1^{e_1+1} A_2^{e_2+ 1})) \]  
is identically $0$
as a polynomial in $x_1, x_2, x_3$, for every $d = 0, \dots, n-1$.
Rewriting this polynomial as 
\begin{gather} \nonumber
t_1 x_1 \tr((x_1 A_1 + x_2 A_2 + x_3 A_1 A_2)^d 
A_1^{e_1+1} A_2^{e_2})  \\ 
\nonumber
+ t_2 x_2 \tr( (x_1 A_1 + x_2 A_2 +  x_3 A_1 A_2)^d A_1^{e_1} A_2^{e_2 + 1}) \\ 
\nonumber
+ t_3 x_3 \tr( (x_1 A_1 + x_2 A_2 +  x_3 A_1 A_2)^d A_1^{e_1 + 1} 
A_2^{e_2 + 1}) 
\end{gather}
and applying Lemma~\ref{lem.binomial}(c) to each term, we obtain
\begin{gather} \nonumber t_1 \sum_{(a, b, c)}
n q^{(e_1+1)(b+c) + \frac{c(c-1)}{2}} 
\binom{d}{a, b, c}_q x_1^{a + 1} x_2^b x_3^c \\ 
\label{e.expansion}
+ t_2 \sum_{(a', b', c')}
n q^{e_1(b'+c') + \frac{c'(c'-1)}{2}} 
\binom{d}{a', b', c'}_q 
x_1^{a' + 1} x_2^{b'} x_3^{c'} \\ 
\nonumber
+ t_3 \sum_{(a'', b'', c'')} 
n q^{(e_1+ 1)(b''+c'') + \frac{c''(c''-1)}{2}} 
\binom{d}{a'', b'', c''} x_1^{a''} x_2^{b''} x_3^{c'' + 1} = 0,
\end{gather}
where the sums are takes over triples of non-negative integers
$(a, b, c)$, $(a', b', c')$ and $(a'', b'', c'')$ satisfying
\[  \begin{array}{l} 
 a + b + c = d \\
 a + c +  e_1 + 1 \equiv 0 \! \! \! \! \! \pmod{n} \\
 b + c + e_2 \equiv 0 \! \! \! \! \! \pmod{n},
 \end{array}  
\begin{array}{l} 
a' + b' + c' = d \\
a' + c' +  e_1 \equiv 0 \! \! \! \! \! \pmod{n} \\
b' + c' + e_2 + 1 \equiv 0 \! \! \! \! \! \pmod{n},
\end{array} 
\begin{array}{l} 
a'' + b'' + c'' = d \\
a'' + c'' +  e_1 + 1 \equiv 0 \! \! \! \! \! \pmod{n} \\
b'' + c'' + e_2 + 1 \equiv 0 \! \! \! \! \! \pmod{n}.
\end{array}  \] 
The expression on the left hand side of~\eqref{e.expansion} is a 
homogeneous polynomial in $x_1, x_2, x_3$ of degree $d + 1$. 
Our element 
$w = (t_1 A_1^{e+1} A_2^{e_2}, t_2 A_1^{e_1} A_2^{e_2 + 1}, 
t_3 A_1^{e_1 + 1} A_2^{e_2 + 1})$ of $W_{e_1, e_2}$ lies 
in $V_{e_1, e_2}$ if and only if this polynomial is identically zero.

To make the conditions the vanishing of this polynomial 
imposes on $t_1, t_2, t_3$ more explicit, let us
examine the coefficient of $x_1^{\alpha} x_2^{\beta} x_3^{\gamma}$ 
(with $d + 1 = \alpha + \beta + \gamma$).
This coefficient is zero unless
$\alpha$, $\beta$ and $\gamma$ are chosen so that 
\begin{equation} \label{e.congruences}
\begin{array}{l}
\alpha + \beta + \gamma \leqslant n \\ 
\alpha + \gamma +  e_1 \equiv 0 \pmod{n} \\
\beta + \gamma + e_2 \equiv 0 \pmod{n}. 
\end{array}
\end{equation}
On the other hand, if $\alpha$, $\beta$ and $\gamma$ satisfy
conditions~\eqref{e.congruences}, then setting
\begin{gather}
\nonumber
d := \alpha + \beta + \gamma -1 \\
\nonumber
a = \alpha - 1, \quad b = \beta, \quad c = \gamma \\
\nonumber
a' = \alpha , \quad b' = \beta -1 , \quad c = \gamma \\
\nonumber
a'' = \alpha, \quad b'' = \beta, \quad c'' = \gamma -1,
\end{gather}
we see that the coefficient of
$x_1^{\alpha} x_2^{\beta} x_3^{\gamma}$ is
\begin{gather}  
\nonumber
t_1 nq^{(e_1+1)(\beta + \gamma) + \frac{\gamma(\gamma-1)}{2}} 
\binom{d}{\alpha-1, \beta, \gamma}_q   
+ t_2 nq^{e_1(\beta - 1 + \gamma) + \frac{\gamma(\gamma -1)}{2}} 
\binom{d}{\alpha, \beta -1, \gamma}_q \\
\nonumber
+ t_3 nq^{(e_1 + 1)(\beta + \gamma - 1) + \frac{(\gamma -2)(\gamma - 1)}{2}} 
\binom{d}{\alpha, \beta, \gamma -1}_q \, . 
\end{gather}
Equating this coefficient to $0$ and
dividing through by $n q^{e_1(\beta + \gamma) + \frac{\gamma(\gamma -1)}{2}}$, 
we obtain
\begin{equation} \label{e.constraint1}
t_1 q^{\beta + \gamma} 
\binom{d}{\alpha-1, \beta, \gamma}_q +   
t_2 q^{-e_1} 
\binom{d}{\alpha, \beta -1, \gamma}_q \\
+ t_3 q^{\beta - e_1}
\binom{d}{\alpha, \beta, \gamma - 1}_q \\
= 0 
\end{equation}
In summary, $w = (t_1 A_1^{e+1} A_2^{e_2}, t_2 A_1^{e_1} A_2^{e_2 + 1}, 
t_3 A_1^{e_1 + 1} A_2^{e_2 + 1})$ lies in $V_{e_1, e_2}$ 
if and only if~\eqref{e.constraint1} holds 
for every $\alpha, \beta, \gamma$ satisfying 
conditions~\eqref{e.congruences}.  

\begin{proof}[Proof of Proposition~\ref{prop.final}(a)]
Our goal is to show that $w = (t_1 A_1, t_2 A_2, t_3 A_1 A_2)$ lies
in $V_{0, 0}$ if and only if 
$t_1 = t_2 = t_3 = 0$.  Note that here $e_1 = e_2 = 0$, 
and $(\alpha, \beta, \gamma) = (n, 0, 0)$, 
$(0, n, 0)$, $(0, 0, n)$ satisfy conditions~\eqref{e.congruences}. 
Substituting $(\alpha, \beta, \gamma) = (n, 0, 0)$
into~\eqref{e.constraint1}, and remembering that
{\Large $\binom{d}{a, b, c}_q$}$ = 0$ whenever $a$, $b$ 
or $c$ is $< 0$, we obtain
\[ t_1 \binom{n - 1}{n-1, 0, 0}_q = 0, \] or equivalently, $t_1 = 0$.  
Similarly, setting $(\alpha, \beta, \gamma) = (0, n, 0)$ yields
$t_2 = 0$, and setting $(\alpha, \beta, \gamma) = (0, 0, n)$ yields
$t_3 = 0$. This proves part (a).
\end{proof}

\begin{proof}[Proof of Proposition~\ref{prop.final}(b)]
Here $(e_1, e_2) \neq (0, 0)$, and we can
use Lemma~\ref{lem.ratios} to simplify formula~\eqref{e.constraint1}
as follows
\[  t_1 q^{\beta + \gamma}(1-q^{\alpha}) + 
 t_2 q^{-e_1} (1 - q^{\beta}) 
 + t_3 q^{\beta - e_1}(1 - q^{\gamma}) = 0 \, . \] 
Using~\eqref{e.congruences}, we can rewrite this in a more symmetric way,
as
\begin{equation} \label{e.constraint2}
t_1 (q^{-e_2} - q^{d+1}) + t_2 (q^{-e_1} - q^{d+1}) 
+ t_3 (q^{d+1} - q^{-e_1 - e_2}) = 0 \, , 
\end{equation}
where $d + 1 = \alpha + \beta + \gamma$, as before.

\smallskip
\begin{claim} Suppose $e_1, e_2 = 0, \dots, n-1$ and $(e_1, e_2) \neq (0, 0)$.
Then there exist triples of non-negative integers, 
$(\alpha_1, \beta_1, \gamma_1)$ and $(\alpha_2, \beta_2, \gamma_2)$ 
satisfying conditions~\eqref{e.congruences}
such that $d_1 \not \equiv d_2 \pmod{n}$. 
Here $d_1 = \alpha_1 + \beta_1 + 
\gamma_1 - 1$ and $d_2 = \alpha_2 + \beta_2 + \gamma_2 - 1$. 
\end{claim}

\smallskip
We will now deduce Proposition~\ref{prop.final}(b) from this claim.
The proof of the claim will be deferred to the end of this section.
Assuming the claim is established, formula~\eqref{e.constraint2} tells us that
if $(t_1 A_1^{e_1+1} A_2^{e_2}, t_2 A_1^{e_1} A_2^{e_2 + 1}, t_3 A_1^{e_1 + 1}
A_2^{e_2 + 1})$ lies in $V_{e_1, e_2}$, then 
$t_1, t_2$ and $t_3$ satisfy the linear equations
\begin{gather}
\nonumber
\text{$ t_1 (q^{-e_2} - q^{d_1+1}) + t_2 (q^{-e_1} - q^{d_1+1}) 
+ t_3 (q^{d_1+1} - q^{-e_1 - e_2}) = 0$,} \\
\label{e.last-system}
\text{$ t_1 (q^{-e_2} - q^{d_2+1}) + t_2 (q^{-e_1} - q^{d_2+1}) 
+ t_3 (q^{d_2+1} - q^{-e_1 - e_2}) = 0$.} 
\end{gather}
The matrix of this system
\[ \begin{pmatrix}
q^{e_2} - q^{d_1+1} & q^{-e_1} - q^{d_1+1} & q^{d_1+1} - q^{-e_1 - e_2} \\
q^{e_2} - q^{d_2+1} & q^{-e_1} - q^{d_2+1} & q^{d_2+1} - q^{-e_1 - e_2} 
\end{pmatrix} \]
is easily seen to have rank $2$. Indeed, the determinants of the
$2 \times 2$ minors are 
\begin{gather} 
\nonumber
\pm (q^{d_1+1} - q^{d_2+1})(q^{-e_2} - q^{-e_1}), \\
\nonumber
\pm (q^{d_1+1} - q^{d_2+1})(q^{-e_1-e_2} - q^{-e_1}), \quad \text{and}\\
\nonumber
\pm (q^{d_1+1} - q^{d_2+1})(q^{-e_1 - e_2} - q^{-e_1}) \, .  
\end{gather}
Since $q^{d_1 +1} \neq q^{d_2 + 1}$, 
all three of these determinants can only be zero if 
$q^{-e_1} = q^{-e_2} = q^{-e_1 - e_2}$ or equivalently, 
$e_1 \equiv e _2 \equiv e_1 + e_2
\pmod{n}$, i.e., $(e_1, e_2) = (0, 0) \pmod{n}$,
contradicting our assumption that $(e_1, e_2) \neq (0, 0)$. 
We conclude that the solution space to
system~\eqref{e.last-system} is of dimension $ \leqslant 1$
and consequently, $\dim(V_{e_1, e_2}) \leqslant 1$
On the other hand, by Remark~\ref{rem.commutator}, 
$\dim(V_{e_1, e_2}) \geqslant 1$.
This shows that $\dim(V_{e_1, e_2}) = 1$, 
thus completing the proof of Proposition~\ref{prop.final}(b).

We now turn to the proof of the claim. The statement of the claim 
is clearly symmetric with respect to $e_1$ and $e_2$. That is, 
if the triples \[ \text{$(\alpha_1, \beta_1, \gamma_1)$ and
$(\alpha_2, \beta_2, \gamma_2)$} \] satisfy the claim 
for $(e_1, e_2)$, then the triples $(\beta_1, \alpha_1, \gamma_1)$,
$(\beta_2, \alpha_2, \gamma_2)$ will satisfy the claim for
$(e_2, e_1)$. Thus for the purpose of proving this claim, we may 
assume without loss of generality that 
$0 \leqslant e_2 \leqslant e_1 \leqslant n-1$.

\smallskip
{\bf Case 1:} $e_2 \geqslant 1$. Here the triples 
\[ \text{$(\alpha_1, \beta_1, \gamma_1) = (0, e_1 - e_2, n - e_1)$ 
and
$(\alpha, \beta, \gamma) = (1, e_1 - e_2 + 1, n - e_1 - 1)$} \] 
satisfy conditions~\eqref{e.congruences} and yield distinct sums
$d_1 + 1 = \alpha_1 + \beta_1 + \gamma_1 = n - e_2$ and
$d_2 + 1 = \alpha_2 + \beta_2 + \gamma_2 = n - e_2 + 1$.
Note that $d_2 + 1 \leqslant n$, because we are assuming 
that $e_2 \geqslant 1$.

\smallskip
{\bf Case 2:} $e_2 = 0$ but $1 \leqslant e_1 \leqslant n-1$. Set 
$(\alpha_1, \beta_1, \gamma_1) = (0, e_1, n - e_1)$, as in
Case 1,  and $(\alpha_2, \beta_2, \gamma_2) = (n - e_1, 0, 0)$.
Then $d_1 + 1 = n$ and $d_2 + 1 = n - e_1$ are, once again, distinct
modulo $n$.
This completes the proof of the claim and hence, 
of Proposition~\ref{prop.final} and of Theorem~\ref{thm.main}.
\end{proof}

\section{The case where $r \geqslant n^2 -1$}
\label{sect.frobenius}
  
Let $K_{r, n} := k(\Mat_n^r)^{\PGL_n}$ is the field of matrix invariants
and $K_{r, n}'$ is the subfield generated by the coefficients of the
generalized characteristic polynomial
\[ (A_1, \dots, A_r) \quad \mapsto \quad 
\det(x_0 I + x_1 A_1 + \dots + x_r A_r) \, , \]
as in Section~\ref{sect.r=3}. Recall that $K_{r, n}$ is
the field of rational functions on $\Mat_n^r /\! \!/\PGL_n$
and $K_{r, n}'$ is the field of rational functions on $\DHyp_{r, n}$.

By abuse of notation we will denote by $t$ the transposition map
$\Mat_n \to \Mat_n$ as well as the maps it induces on
$\Mat_n^r$ (by applying $t$ to each component),
$\Mat_n^r /\!\!/\PGL_n$, and their function fields.
For example, 
\[ t \left( \tr (A_1A_2 A_3) \right) := 
 \tr(A_1^t A_2^t A_3^t) = \tr(A_3 A_2 A_1) \, . 
\] 
Since $ \det(x_0 I + x_1 A_1 + \dots + x_r A_r) =
\det(x_0 I + x_1 A_1^t + \dots + x_r A_r^t)$, we have 
\begin{equation} \label{e.Krn} 
K_{r, n}' \subset K_{r, n}^t \, .
\end{equation}
Our standing assumption that the base field $k$ is algebraically closed
of characteristic $0$ or $> n$ remains in force.

\begin{lemma} \label{lem.conditions}
Assume $r \geqslant 2$, $n \geqslant 2$ and $(r, n) \neq (2, 2)$.
Then the following assertions are equivalent.

\smallskip
(a) The general fiber of 
$\overline{P} \colon \Mat_n^r /\! \!/\PGL_n \to \DHyp_{r, n}$
consists of exactly two points corresponding to the conjugacy 
classes of $(A_1, \dots, A_r)$ and $(A_1^t, \dots, A_r^t)$.

\smallskip
(b) $[K_{r, n} : K_{r, n}'] = 2$.

\smallskip
(c) $K_{r, n}' = K_{r, n}^t$.
\end{lemma}

\begin{proof} 
(a) $\Longrightarrow$ (b).
Theorem~\ref{thm.main} tells us that $K_{n, r}/K_{n, r}'$ is a finite 
separable extension. Thus the general fiber of $\overline{P}$
consists of exactly $[K_{r, n} : K_{r, n}']$ points. 

\smallskip
(b) $\Longleftrightarrow$ (c).
Under our assumptions on $r$ and $n$, $t$ is an automorphism of
$K_{r, n}$ of order $2$. Thus 
$[K_{r, n} : K_{r, n}^{t}] = 2$.
In view of~\eqref{e.Krn},
$[K_{r, n} : K_{r, n}'] \geqslant 2$, 
and equality holds if and only if $K_{r, n}'= K_{r, n}^t$.

\smallskip
(c) $\Longrightarrow$ (a).
If (c) holds, then a general fiber of $\overline{P}$ 
has exactly two elements. If such a fiber contains a point 
representing $A$, it also contains a point representing $A^t$.
For $A \in \Mat_n^r$ in general position, these points 
are distinct (here we are using the assumption that $(r, n) \neq (2, 2)$!),
so there cannot be any others.
\end{proof}

Our goal now is show that in the case where $r \geqslant n^2 - 1$, 
Theorem~\ref{thm.main} can be strengthened as follows. 

\begin{theorem} \label{thm.frobenius}
The equivalent conditions of Lemma~\ref{lem.conditions}
hold if $r \geqslant n^2 - 1$, for any $n \geqslant 2$.
\end{theorem}

The rest of this section will be devoted to proving 
Theorem~\ref{thm.frobenius}. We proceed in three steps.
(1) Lemma~\ref{lem.frobenius2} settles the case, where $n = 2$,
(2) Lemma~\ref{lem.frobenius} settles the case, where $r = n^2 - 1$, and
(3) Proposition~\ref{prop.r+++} supplies the induction step, showing that
if the equivalent conditions of Lemma~\ref{lem.conditions}
hold for some parameters $r$ and $n$, then they also hold for
$r+ 1$ and $ n$, provided that $r, n \geqslant 3$.

\begin{lemma} \label{lem.frobenius2} Assume $r \geqslant 2$. Then

\smallskip
(a) $K_{r, 2}' = k(\tr(A_i), \tr(A_i A_j) \, | \, i, j = 1, \dots, r)$.

\smallskip
(b) $K_{r, 2}' = K_{r, 2}^t$.
\end{lemma}

\begin{proof} (a) Recall that  
$K_{r, n}'$ is generated over $k$ by the coefficients of 
$\det(x_0 I + x_1 A_1 + \dots + x_r A_r)$, where $I$ is the $2 \times 2$ 
identity matrix.
Setting $X := x_0I + x_1 A_1 + \dots + x_r A_r$ and
using the formula $\det(X) = \dfrac{1}{2}(\tr(X)^2 - \tr(X^2))$, we 
see that $K_{r, 2}'$ is generated over $k(\tr(A_i) \, | \, i = 1, \dots, r)$
by the coefficients of $\tr(X^2)$, and part (a) follows.

(b) Let $V$ be the $3$-dimensional subspace of trace zero $2 \times 2$ 
matrices, equipped with the non-degenerate quadratic form $q(A, B) = \tr(AB)$.
Then the representation $\PGL_2 \to \GL(V)$ given by the conjugation
action is an isomorphism between $\PGL_2$ and $\SO(V) \simeq \SO_3$.
The transposition map $t \colon V \to V$ also preserves 
the trace form;
the subgroup $G$ of $\GL(V) \simeq \SO_3$ generated by 
$\PGL_2$ and $t$ is easily seen to be the full orthogonal group
$\Orth(V)$. 
Now observe that by definition,
$K_{r, 2}^t = k(\Mat_2^r)^G$. Let us identify
$\Mat_2$ with $V_0 \oplus V$, via the isomorphism
\[ A \to (\tr(A), A - \frac{1}{2} \tr(A)) \, . \] 
Here $V_0$ denotes the $1$-dimensional
trivial representation of $G$. 
This identifies $K_{r, 2}^{t}$ 
with the field of $\Orth(V)$-invariants of 
$V_0^r \oplus V^r$.  The First Fundamental Theorem of classical 
invariant theory tells us that the field of invariants 
is generated by $k(V_0^r)$ and the functions
\[ (t_1, \dots, t_r, v_1, \dots, v_r) \mapsto q(v_i, v_j) \, , \]
where $t_1, \dots, t_r \in V_0$, $v_1, \dots, v_r \in V$; see, 
e.g.,~\cite[Theorem 5.7]{dCP}.  Remembering
our identification between $\Mat_n$ and $V_0 \oplus V$, 
we readily translate this into
\[ K_{r, 2}^t = k(\tr(A_i), \tr(A_i A_j) \, | \, i, j = 1, \dots, r). \]
The desired equality, $K_{r, 2}' = K_{r, 2}^t$ now follows from part (a).
\end{proof}

\begin{lemma} \label{lem.frobenius} Let $r = n^2 - 1$ and assume that
$I_1, A_1, \dots, A_r$ span $\Mat_n$ as a $k$-vector space. If 
\[ \det(x_0 I + x_1 A_1 + \dots + x_r A_r) = 
\det(x_0 I + x_1 B_1 + \dots + x_r B_r) \]
for some $B = (B_1, \dots, B_r) \in \Mat_n^r$, then $B$ is conjugate to
$A$ or $B$ is conjugate to $A^t$.
\end{lemma}

\begin{proof} Let $T \colon \Mat_n \to \Mat_n$ be the linear transformation
taking $I$ to $I$ and $A_i$ to $B_i$ for every $i = 1, \dots, r$. 
By our assumption $T$ preserves the determinant function. By a theorem
of Frobenius, there exist $P, Q \in \Mat_n$ such that $\det(P) \det(Q) = 1$
and  $T(X) = C X D$; see the references in Remark (1) in the Introduction.
Since $T(I) = I$, we have $C = D^{-1}$, and the lemma follows.
\end{proof}

\begin{proposition} \label{prop.r+++} 
Assume $r, n \geqslant 3$. 
If $K_{r}' = K_{r, n}^t$, then $K_{r + 1}' = K_{r+1, n}^t$.
\end{proposition}

\begin{proof} This proposition is in the same spirit as
Proposition~\ref{prop.r++}, and we will use a more elaborate 
version of the same argument. Once again, a key ingredient
will be supplied by Lemma~\ref{lem.procesi}, which asserts 
that there exist finitely many monomials
$M_1, \ldots, M_N$ in $A_1$ and $A_2$ such that $K_{r, n}$ is generated, 
as a field extension of $k$, by the elements 
$\tr(M_i)$ and $\tr(M_i A_j)$, where $i = 1, \dots, N$, and $j = 3, \dots, r$.
To simplify the notation, set
\begin{gather} 
\nonumber
s_i := \tr(M_i) + \tr(M_i)^t, \\
\nonumber
\Delta_i := \tr(M_i) - \tr(M_i)^t, \\
\nonumber
s_{i, j} := \tr(M_i A_j) + \tr(A_j M_i)^t, \\
\nonumber
\Delta_{i, j} := \tr(M_i A_j) - \tr(A_j M_i)^t.
\end{gather}
We will also need a non-zero element $f \in K_{2, n}$ with the property
that $t(f) = -f$. Such an element exists for every $n \geqslant 3$;
for example, we can take
\[ f(A_1, A_2) := \tr(A_1A_2A_1^2 A_2^2) -
\tr(A_2^2 A_1^2 A_2 A_1). \]
For this choice of $f$, the equality $t(f) = - f$ is clear;
the computation on~\cite[p.~72]{reichstein} shows that $f \neq 0$.
(Note that here we are using the assumption that $n \geqslant 3$.
For $n = 2$, $f$ cannot exist because $t$ acts trivially on $K_{2, n}$,
and our argument below breaks down.  This is the reason
we handled the case where $n = 2$ separately, 
in Lemma~\ref{lem.frobenius2}.) Now
\begin{align} 
\nonumber
K_{r+1, n}^t & = & k(\tr(M_i), \tr(M_i A_j) \, | \, i = 1, \dots, N, \,
j = 3, \dots, r+1 ) \\ 
\nonumber
   & = & k(s_i, \Delta_i, s_{ij}, \Delta_{ij}) \, 
| \, i = 1, \dots, N, \,   j = 3, \dots, r+1 )^t \\ 
\nonumber
 & = & k(s_i, \Delta_i f , s_{ij}, \Delta_{ij}f, f ) \, 
| \, i = 1, \dots, N, \,   j = 3, \dots, r+1 )^t  
\end{align}
The elements
$s_i$, $\Delta_i f$ , $s_{ij}$, $\Delta_{ij}f$ are all fixed by
$t$, while $t(f) = - f$. Thus
\begin{equation} \label{e++}
K_{r+1, n}^t = k(s_i, \Delta_i f , s_{ij}, \Delta_{ij}f, f^2 ) . 
\end{equation}
Clearly $K_{r+1, n}' \subset K_{r+1, n}^t$. To prove equality, 
it suffices to show that each of the generators
$s_i, \Delta_i f , s_{ij}, \Delta_{ij}f$ and $f^2$ lie in $K_{r+1, n}'$.
 
Note that $s_i$, $\Delta_i f$ and $f^2$ lie in $K_{2, n}^t$, and
$s_{i3}$ and $\Delta_{i3} f$ lie in $K_{3, n}^t$. Since $r \geqslant 3$, 
these elements all lie in $K_{r, n}^t$. By our assumption, 
$K_{r, n}^t = K_{r, n}' \subset K_{r+1, n}'$. Hence, each of the generators
$f^2, s_i, \Delta_i f , s_{i3}, \Delta_{i3}f$ lie in $K_{r+1, n}'$. By symmetry, 
$s_{ij}$ and $\Delta_{ij}f$ also lie in $K_{r+1, n}'$, for any
$j = 3, \dots, r + 1$. We conclude that
$f^2, s_i, \Delta_i f , s_{ij}, \Delta_{ij}f$ all lie in
$K_{r+1, n}'$. By~\eqref{e++}, $K_{r+1, n}^t = K_{r+1, n}'$, as desired.
\end{proof}

\bibliographystyle{amsalpha} 
\def\cprime{$'$}
\providecommand{\bysame}{\leavevmode\hbox to3em{\hrulefill}\thinspace}
\providecommand{\MR}{\relax\ifhmode\unskip\space\fi MR }
\providecommand{\MRhref}[2]{%
  \href{http://www.ams.org/mathscinet-getitem?mr=#1}{#2}
}
\providecommand{\href}[2]{#2}

\end{document}